\documentclass[12pt,a4paper]{article}
\usepackage[utf8]{inputenc}
\usepackage{amsmath}
\usepackage{amsfonts}
\usepackage{amssymb}
\usepackage{amsthm}
\usepackage{amssymb}
\usepackage{graphicx}

\theoremstyle{plain}

\newtheorem{theor}{Theorem}[section]

\newtheorem{prop}[theor]{Proposition}
\newtheorem{defn}[theor]{Definition}

\newtheorem{cor}[theor]{Corollary}
\newtheorem{lemma}[theor]{Lemma}
\newtheorem{rem}[theor]{Remark}

\title{Minimal dispersion of large volume boxes in the cube}
\author{Kurt S. MacKay}

\newcommand\address{\noindent\leavevmode

\medskip
\noindent
Kurt  S. MacKay\\
Dept.~of Math.~and Stat.~Sciences,\\
University of Alberta, \\
Edmonton, AB, Canada, T6G 2G1.\\
\texttt{\small
e-mail:  ksmackay@protonmail.com}
}
\begin{document}
\maketitle
\begin{abstract}
%%%%%%%%%%%%%%%%%%%%%%%%%%%%%
In this note we present a construction which improves the best known bound on the minimal dispersion of large volume boxes in the unit cube.
Let $d>1$. 
The dispersion of $T \subset [0,1]^d$ is defined as the supremum of the volume taken over all axis parallel boxes in the cube which do not intersect~$T$.
The minimal dispersion of $n$ points in the cube is defined as the infimum of the dispersion taken over all $T$ such that $|T| = n$.
Define the ``large volume" regime as the class of all volumes $\frac{1}{4} < r \leq \frac{1}{2}$. 
The inverse of the minimal dispersion is denoted as $N(r,d)$.
%In this setting
When the volume is large, the best known upper bound on $N(r,d)$ is of the order 
$ {(r - \frac{1}{4})}^{-1}.$
%
%Let $r$  be a large volume,% \in (\frac{1}{4},\frac{1}{2}]$,
 %The construction given in this note yields an
The construction presented in this note yields an
 %presented in this paper yields a
 % dimension independent 
 upper bound given by 
$$N(r,d) \leq \left \lfloor  \frac{\pi}{\sqrt{r - \frac{1}{4}}} \right \rfloor  - 3 .$$
%
%We also show that
Some of our intermediate estimates are sharp given the condition that $d \geq C_{r}$, where $C_r$ is a positive constant which depends only on the volume~$r$.

\end{abstract}

\section{Introduction}
The dispersion of $T \subset [0,1]^d$ is defined as the supremum of the volume taken over all axis parallel boxes in the cube which do not intersect~$T$.
The minimal dispersion of $n$ points in the cube is defined as the infimum of the dispersion taken over all $T$ such that $|T| = n$.
The inverse of the minimal dispersion, denoted as $N(r,d)$, is the size of the smallest set of points which intersects any axis parallel box with volume exceeding $r$.
The problem of estimating the minimal dispersion (which was originally defined in \cite{UM} as modification of a concept in \cite{EH}) has been given considerable attention in recent years. Some contemporary works such as
\cite{C}, \cite{BA}, \cite{D}, \cite{AD}, \cite{AP} \cite{KR}, \cite{LV}, \cite{LL}, \cite{NR}, \cite{PR}, \cite{RT}, \cite{S}, \cite{UV}, \cite{UV1}, are focused on this problem, or modifications (general $k$-dispersion, spherical dispersion, etc.) thereof. We will refer to these works when we discuss the known results, and the best known bounds on the minimal dispersion. The dispersion of some particular sets has been studied in \cite{K}, \cite{T}, \cite{U}. Dispersion is of interest in studying topics in subjects such as Discrete Geometry and Approximation Theory and in studying particular random point configurations as in \cite{RT}. Define the ``large volume" regime as the class of all volumes $\frac{1}{4} < r \leq \frac{1}{2}$.  When the volume is large, the best known upper bound given in \cite{S},
is of the order $ {(r - \frac{1}{4})}^{-1}.$
In this note we present a construction which improves the best known bound in this setting.
First, consider the situation when $r \geq \frac{1}{2}$. The minimal dispersion is attained with one point at the center of the cube.
Thus, it is clear that in the large volume regime we are interested in estimating the minimal dispersion when $ r < \frac{1}{2}$.
Theorem \ref{thm1} improves the best known bound previously given by Sosnovec in \cite{S}.
\begin{theor}\label{thm1} 
Let $d > 1$, and let $r \in (\frac{1}{4},\frac{1}{2})$. 
Then 
\begin{equation}\label{eq1}
N(r,d) \leq \left \lfloor  \frac{\pi}{\sqrt{r - \frac{1}{4}}} \right \rfloor  - 3.
\end{equation}
\end{theor}
Let $r \in (\frac{1}{4},\frac{1}{2})$, and let $\textbf{1} = (1,1,\ldots,1) \in \mathbb{R}^d$. 
We construct a set of points on the diagonal in the following way. Consider the fractional sequence 
$$
\mathfrak{Q}(r) =  \bigg\{  r , \  \frac{r}{1 - r}, \  \frac{r}{1 - \frac{r}{1-r}} \ ,  \ 
\frac{r}{1 - \frac{r}{1-\frac{r}{1-r}}},  \ 
\frac{r}{1 - \frac{r}{1-\frac{r}{1-\frac{r}{1-r}}}}, \ 
 \ldots \bigg\}
.$$
Denote the subsequent elements in the sequence as $q_1,q_2, \ldots  \in \mathfrak{Q}(r)$.
We show that there exists a smallest number $n: = n_r \geq 1$  such that $$q_n \geq 1 - r.$$
Consider the configuration given by $$ \mathfrak{q}(r) = \{q_i \mathbf{1} : 1\leq i \leq n \}.$$ 
\newpage 
\noindent
We can see some examples of the configurations when $|\mathfrak{q}(r)| = 5,12,19,26$ (respectively). They are given below. \\
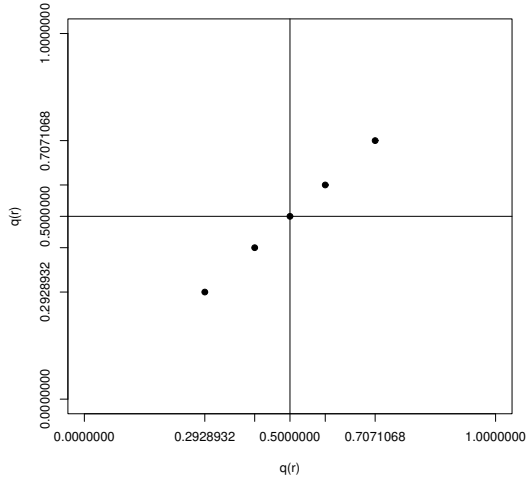
{\includegraphics[scale = .40,trim={5cm 1cm 5cm 1cm},clip]{c1.png}}
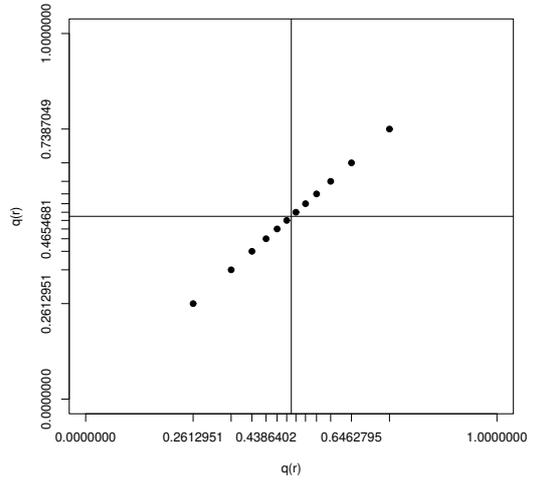
{\includegraphics[scale = .40,trim={5cm 1cm 5cm 1cm},clip]{c2.png}}\\
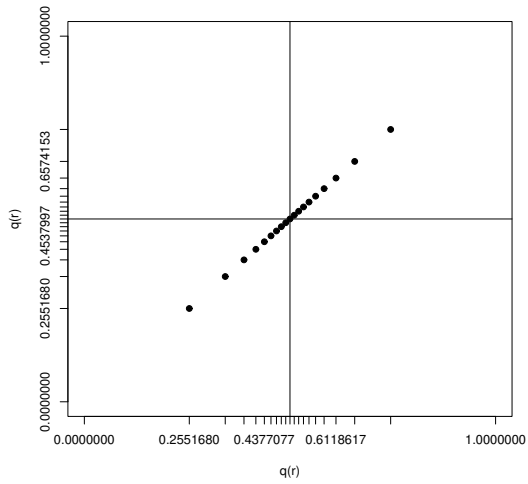
{\includegraphics[scale = .40,trim={5cm 1cm 5cm 1cm},clip]{c3.png}}
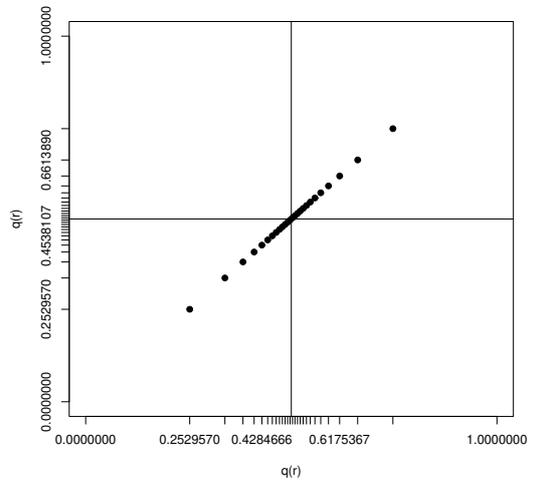
{\includegraphics[scale = .40,trim={5cm 1cm 5cm 1cm},clip]{c4.png}}
\newpage
\noindent
Define the function 
\begin{equation}\label{step}
r \longmapsto|\mathfrak{q}(r)|
\end{equation}
where $|\cdot|$ denotes the cardinality. 
\noindent
\\
It may not be obvious yet, but we can see what the function looks like (a monotone decreasing step function) on some inclusion-ordered intervals in the examples given below. \\
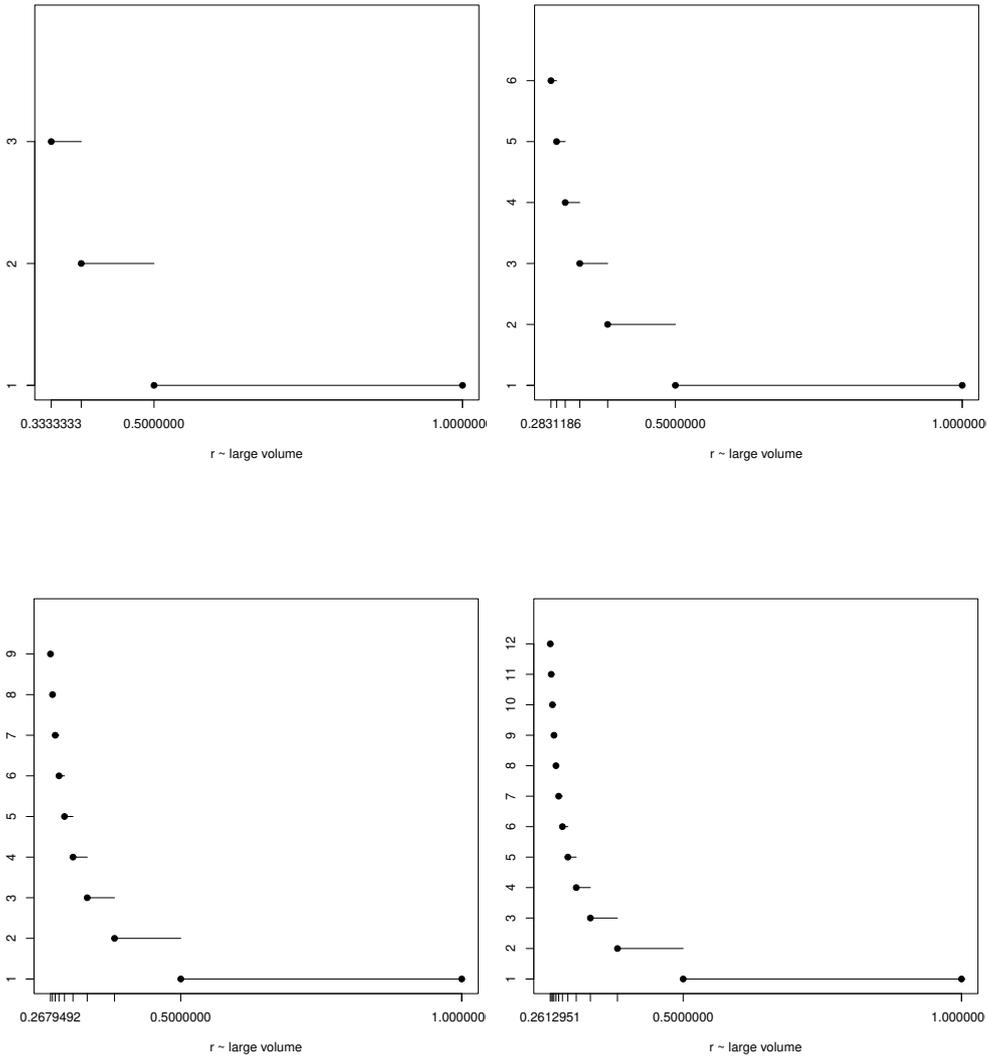
{\includegraphics[scale = .40,trim={5.925cm 1cm 5.923cm 1cm},clip]{1.png}}
{\includegraphics[scale = .4,trim={5.922cm 1cm 5.921cm 1cm},clip]{2.png}}\\
{\includegraphics[scale = .4,trim={5.95cm 1cm 5.921cm 1cm},clip]{3.png}}
{\includegraphics[scale = .4,trim={5.922cm 1cm 5.93cm 1cm},clip]{4.png}}
\\
We obtain Theorem \ref{thm1} by showing that this function is an upper bound for the minimal dispersion.
Finally, we show that some of our estimates are sharp given that $d \geq C_{r}$, where $C_r$ is a positive constant dependent on $r$.
\section{Previous results}
%%%%%%%%%%%%%%%%%%%%%%%%%%%%%
Recall that the inverse of the minimal dispersion, denoted as $N(r,d)$ is defined as the smallest number of points that are needed to intersect any axis parallel box whose volume exceeds $r$.  

In their work \cite{C}, Aistleitner, Hinrichs, and Rudolf proved a lower bound for $r < \frac{1}{4}$ given by
 \begin{equation} \label{eq:R1}
(1 - 4r)\frac{\log_{2}d}{4r} \leq N(r,d).
\end{equation}
This result gives a non-trivial estimate showing that the dispersion asymptotically increases with dimension 
when the volume is not large.
The upper bound given by
\begin{equation} \label{eq:R3}
 N(r,d)\leq \frac{2^{7d + 1}}{r}
\end{equation}
was proved by Larcher, and is also presented in \cite{C}.
This is an improvement on the bound given by Rote and Tichy in \cite{RT}.
The upper bound given by
\begin{equation} \label{eq:R2}
N(r,d) \leq \frac{8d}{r} \log_{2} \bigg(\frac{33}{r} \bigg)
\end{equation}
is a consequence of a more general result given in \cite{BE}.
The authors present an argument which uses the $VC-$dimension of $\mathfrak{B}$ (which is $2d$) instead of the ambient dimension $d$.
In the paper \cite{Ru}, Rudolf presented a probabilistic argument which yields the bound in (\ref{eq:R2}).
The bound in (\ref{eq:R2}) is an improvement on (\ref{eq:R3}) under the assumption that $r \geq \exp(-C^d)$ where $C > 1$ is an absolute constant.
Sosnovec in \cite{S} obtained another upper bound which is better when $d$ grows to infinity, namely
\begin{equation}
N(r,d) \leq C_r \log_2d.
\end{equation}
The constant $C_r$ obtained in \cite{S} grows extremely fast with $r$.
This constant was improved by Ullrich and Vyb\`iral \cite{UV} who showed that $$C_r = \frac{2^7}{r^2} \log_{2}^2 \big(\frac{1}{r} \big).$$
Litvak in \cite{LV} gives an improvement on the known bounds when $r \leq \exp(-d)$, and showed that
$$ N(r,d) \leq \frac{C \ln d}{r} \ln \big(\frac{1}{r} \big).$$
Litvak also established that when $r \geq (\ln^2d)/(d \ln \ln(2d))$ that
$$N(r,d) \leq \frac{C \ln d}{r^2} \ln \big(\frac{1}{r} \big),$$
which is an improvement on the bound given by Ullrich and Vyb\`iral.
Recently, Litvak and Livshyts in \cite{LL}, proved an upper bound for $d\geq 2$, and $r \in (0,\frac{1}{2}]$ given by 
\begin{equation} \label{eq:LL}
N(r,d) \leq 12e \frac{4d \log( \log(\frac{8}{r})) + \log(1/r)}{r}.
\end{equation}
They also showed that a random choice of points with respect to the uniform distribution in the cube gives the result with high probability.
We mention here that Hinrichs, Krieg, Kunsch, and Rudolph in \cite{AD} showed that using a random choice of points uniformly distributed within the cube, one cannot expect to get anything better than
$$\max \bigg\{  \frac{c}{r} \log \bigg(  \frac{1}{r} \bigg), \frac{d}{2r} \bigg\}.$$
Thus we can see that in the probabilistic setting, (\ref{eq:LL}) is close to the best possible.
Now we turn our attention to the large volume regime $r > \frac{1}{4}$.
Sosnovec in \cite{S} gave a dimension independent upper bound for $N(r,d)$.
In particular, he proved that
for  $r \in (\frac{1}{4}, 1)$, and $d \geq 2$, $$N(r , d)  \leq \left \lfloor {\frac{1}{r - \frac{1}{4}} } \right \rfloor + 1.$$
\noindent
The goal of this paper is to improve this bound. This is established in Theorem \ref{thm1}.
\subsection{Definitions and Notation}
Let $d \geq 1$. 
Denote unit cube as $[0,1]^d$.
Let $|\cdot|$ denote the cardinality, let $\textrm{dom}(\cdot)$ denote the domain,
and let $\textbf{1} := (1,1,\ldots,1) \in \mathbb{R}^d$.
The set of axis parallel boxes is defined as
$$\mathfrak{B} := \bigg \{ {\prod}_{i = 1}^{d} I_i: \  I_i = [a_i,b_i) \subset [0,1]  \bigg \}.$$
%$$\mathfrak{B} = \{B = I_1\times I_2 \times \ldots \times I_d: \  I_k \subset [0,1], \ 1 \leq k \leq d  \}.$$
%
%
The dispersion of $T \subset [0,1]^d$ is defined as $$\textrm{disp}(T) := \sup_{B\in \mathfrak{B}, \ B\cap T = \emptyset}\textrm{Vol}(B).$$
The minimal dispersion is defined as $$\textrm{disp}^{*}(n,d) := \inf_{T \subset [0,1]^d, \  |T| = n}\textrm{disp}(T).$$
Let $r \in [0,1]$. The inverse of the minimal dispersion is defined as $$N(r,d) := \min\{n \in \mathbb{N}:\textrm{disp}^{*}(n,d) \leq r\}.$$
Let $k \geq 0$.
Inductively define a sequence of functions $\{f_k\}_{k \geq 0}$ in the following way.
Let $\beta_0 = 1$.
Define $f_0:[0,1]\rightarrow[0,1]$ as the identity $f_0(x) = x$.
Given  functions $f_0, f_1, \ldots f_{k-1}$, and numbers $\beta_0,\beta_1, \ldots , \beta_{k-1}$, define
\begin{equation}\label{beta}
\beta_k = \inf\{x  \geq 0: x \in \textrm{dom}(f_{k-1}),\  f_{k-1}(x) = 1 \},
\end{equation}
and define $f_k: [0,\beta_k) \rightarrow \mathbb{R}$ by
\begin{equation}\label{f}
f_k(x) = \frac{x}{1 - f_{k-1}(x)}.
\end{equation}
Proposition \ref{pro1} shows that the infimum in (\ref{beta}) is attained for all~$k~\geq~0$.
%We will see that the numbers in (\ref{beta}) are the endpoints of the interval partition described in the introduction.
Let $k \geq 0$.
It is clear that $\textrm{dom}(f_k) \subset \textrm{dom}(f_{k-1})$, hence  $\beta_{k} < \beta_{k-1}$.
Let $r \in (\frac{1}{4},\frac{1}{2}]$ (we include the the right endpoint here since the intermediate estimates are still valid here).
The following definition is a complete description of the step function introduced in (\ref{step}).
Define
\begin{equation}\label{afunc}
\alpha(r) := \inf\{k \geq 0 :r \in \textrm{dom}(f_{k}),\ f_k(r) \geq 1 - r\} + 1,
\end{equation}
and define
\begin{equation} \label{eq:alpha}
n_r := \alpha(r) - 1.
\end{equation}
In Remark \ref{inf} we show that the infimum in (\ref{afunc}) is attained.
%
%%%%%%%%%%%%%%%%%%%%%%%%%%%%%
\section{Auxiliary  tools}\label{aux}

\begin{prop}\label{inc1}
Let $i  \geq 0$. The function $f_i$ is strictly increasing on its domain.
\end{prop}

\begin{proof}
Employ induction on $i$. 
Let $i = 0$.
By definition $\mathrm{dom}(f_0)= [0,1]$. For all $x \in [0,1]$, $f_0(x) = x$.
The base case is seen to be trivial.
Assume that the proposition holds for $ i = 0,1,2,\ldots, k$.
Let $x_1, x_2 \in \textrm{dom}(f_{k+1})$ be such that $x_1 < x_2 $.
Then by domain inclusion $x_1, x_2 \in \textrm{dom}(f_{k})$.
By the induction hypothesis $f_{k}(x_1) < f_{k}(x_2)$. 
It follows from the definition of $f_{k+1}$ as in (\ref{f}) that
$$f_{k+1}(x_1) = \frac{x_1}{1 - f_{k}(x_1)}  < \frac{x_2}{1 - f_{k}(x_2)} = f_{k+1}(x_2).$$
This proves the proposition.
\end{proof}
\noindent
In Proposition \ref{pro1} we show that the infimum in (\ref{beta}) is attained.
\begin{prop}\label{pro1}
Let $r_0 = 1$. 
For each $i \geq 1$ there exists a unique number $$r_i \in \mathrm{dom}(f_i) = [0,\beta_i)$$ with the properties
\begin{equation}\label{P1}
\beta_i = r_{i-1}
\end{equation}
\begin{equation}\label{P2}
f_{i-1}(r_i) = 1 - r_i
\end{equation}
\begin{equation}\label{P3}
 f_{i}(r_i) = 1.
\end{equation}

\end{prop}

\begin{proof}
Note that $\beta_1  = 1 = r_0$.
Employ induction on $n$.
For each $n > 0$ we produce a number $r_n$ with the properties (\ref{P1}), (\ref{P2}), (\ref{P3}).

Let $n = 1$.
Recall the definition of $f_2$ given in (\ref{f}),  by $$f_2(x) = \frac{x}{1 - f_1(x)}.$$
We will show that there exists $r_1 \in$ $\mathrm{dom}(f_1)$ such that $$1 = f_1(r_1) = \frac{r_1}{1 - f_0(r_1)}.$$
Equivalently, we find the solution to the equation $f_0(x) - (1 - x) = 0.$
It is clear that $r_1 = \frac{1}{2} < 1 = \beta_1$, and that $ \beta_2 = r_1$.
This implies that $\mathrm{dom}(f_2) = [0,\beta_2).$ Hence, $r_1$ has properties (\ref{P1}), (\ref{P2}), (\ref{P3}).

Let $n = 2$. 
Recall the definition of $f_3$ given in (\ref{f}), by $$f_3(x) = \frac{x}{1 - f_2(x)}.$$
We show that there exists $r_2 \in$ $\mathrm{dom}(f_2)$ such that $$1 = f_2(r_2) = \frac{r_2}{1 - f_1(r_2)}.$$
Equivalently, we show that there exists a  solution to the equation $$f_1(x) - (1 - x) = 0.$$
The function $f_1$ is strictly increasing by Proposition \ref{inc1}.
Note that $f_1(0) =  0$, and $f_1(r_1) = 1 $.
Apply the Intermediate Value Theorem to $f_1(x) - (1 - x)$. This yields a unique solution $r_2< r_1$ such that
 $$f_1(r_2) - (1 - r_2) = 0.$$
It follows that $ \beta_3 = r_2$, and that $\mathrm{dom}(f_3) = [0,\beta_3).$
Hence $r_2$ has properties (\ref{P1}), (\ref{P2}), (\ref{P3}).

Let $k > 2$. Assume that there exist numbers $r_0,r_1,r_2,\ldots,r_{k-1}$ with the properties (\ref{P1}), (\ref{P2}), (\ref{P3}).
Under the assumption that  $\mathrm{dom}(f_{k}) = [0,\beta_{k})$ where $\beta_{k} = r_{k-1}$,
and  $f_{k-1}(r_{k-1}) = 1$.
Recall the definition of $f_{k+1}$ given in (\ref{f}), by $$f_{k+1}(x) = \frac{x}{1 - f_k(x)}.$$
We show that there exists $r_k \in$ $\mathrm{dom}(f_k)$, such that $$1 = f_k(r_k) = \frac{r_k}{1 - f_{k-1}(r_k)}.$$
Equivalently, we show that there exists a solution to the equation $$f_{k-1}(x) - (1 - x) = 0.$$
The function $f_{k-1}$ is strictly increasing by Proposition \ref{inc1}.
Note that $f_{k-1}(0) =  0,$ and by the induction hypothesis, $f_{k-1}(r_{k-1}) = 1.$
Apply the Intermediate Value Theorem to $f_{k-1}(x) - (1 - x)$.  This yields a unique solution $r_{k} < r_{k-1}$,  such that $$f_{k-1}(r_k) - (1 - r_k) = 0.$$ It follows that $ \beta_{k+1} = r_k$, and that $\mathrm{dom}(f_{k+1}) = [0,\beta_{k+1}).$
This yields a number $r_{k+1}$ with the properties (\ref{P1}), (\ref{P2}), (\ref{P3}).
This proves the proposition.
\end{proof}

\begin{rem}
For each $k \geq 1$, the infimum in the definition of  $\beta_k $ is attained at $\beta_k = r_{k-1}$.
Proposition \ref{pro1} justifies the assertion following the definition in (\ref{beta}).
Fix the sequence
\begin{equation} \label{eq:r}
\{r_m\}_{m \geq 0} =  \{ \beta_{m+1}\}_{m \geq0}.
\end{equation} 
\end{rem}

\begin{prop}\label{pro0}
Let $n \geq 1$. 
Let  $r \in (\frac{1}{4},\frac{1}{2})$ be such that $f_{i}(r) < 1$ for all $i \leq n$.
Then for all $i \leq n$, $$f_{i-1}(r) < f_{i}(r).$$
\end{prop}

\begin{proof}

Fix $n \geq 1 $.
Employ induction on $i$. 
Let $r \in (\frac{1}{4} ,\frac{1}{2})$ be such that $f_{i}(r) < 1$ for all $i \leq n$.
Let $i = 1$.
Recall the definitions of $f_0,$ $f_1$.
Since $r < \frac{1}{2}$, it follows that $$f_0(r) = r < \frac{r}{1-r} = f_1(r).$$
Let $1 \leq k < n$. 
Assume as the induction hypothesis that  for all $1 \leq  i \leq k$, $$f_{i-1}(r) < f_{i}(r).$$ 
By assumption $f_{k-1}(r) < f_{k}(r)$,
then by definition of $f_k$ it follows that
$$f_{k}(r) = \frac{r}{1 - f_{k-1}(r)} < \frac{r}{1 - f_{k}(r)} = f_{k+1}(r).$$
This proves the proposition.
\end{proof}

\begin{cor}\label{cor0}
Let $r \in (\frac{1}{4},\frac{1}{2}]$.
Let $n  \geq 1 $ be such that $r < r_n$. 
Then for all $i \leq n$, $$f_{i-1}(r) < f_{i}(r).$$
\end{cor}

\begin{proof}
Let $r \in (\frac{1}{4},\frac{1}{2}]$.
Let $n  \geq 1$ be such that $r < r_n$.
The sequence $\{r_m\}_{ m >0 }$ defined in (\ref{eq:r}) is decreasing. Hence, $r < r_n < r_{n-1}< \cdots <r_1$.
Property (\ref{P3}) in Proposition \ref{pro1} implies that $f_{k}(r_k) = 1$ for all $k \leq n$.
Since $r < r_k$, it follows by Proposition \ref{inc1} that $ f_{k}(r) < f_{k}(r_k)$.
Hence, 
$$ f_{k}(r) < f_{k}(r_k) = 1.$$
Now apply Proposition \ref{pro0}.
\end{proof}

%%%%%%%%%%%%%%%%%%%%%%%%%%%%
%
\begin{rem}
Corollary \ref{cor0} gives the property that for all $n < k$, $f_n(r_k) < 1$.
Property (\ref{P3}) in Proposition \ref{pro1} gives that for all $k > 0$,  $f_{k}(r_k) = 1$. Let $r \in (\frac{1}{4},\frac{1}{2}]$, and recall the definition  in (\ref{afunc}) given by
$$\alpha(r) := \inf\{k \geq 0 :r \in \mathrm{dom}(f_{k}),\ f_k(r) \geq 1 - r\} + 1.$$
Then we have that for all $k>0$, $$ \alpha(r_k) = k = n_{r_k} + 1.$$ This gives the integral values of the step function in (\ref{step}) evaluated at the left endpoints of the interval partition given by 
\begin{equation}\label{intp}
\cdots[r_3,r_2) , [r_2,r_1),[r_1,1). 
\end{equation}
\end{rem}
The following proposition shows that the function in (\ref{afunc}) is constant over any interval in the partition (\ref{intp}).
\begin{prop}\label{pro6}
Let $i  \geq 1$. 
Let $r \in (\frac{1}{4},\frac{1}{2}]$ be such that $r_{i} \leq r < r_{i-1}$.
Then 
$$\alpha(r) = i .$$
\end{prop}

\begin{proof}
First, let $i = 1$, and let $r \in (\frac{1}{4},\frac{1}{2}]$ be such that $r_{1} \leq r < 1$.
Since $r_1 = \frac{1}{2}$, it follows that $r = \frac{1}{2}$. Hence, $\alpha(r) = 1$.

Let $i  \geq 2$, and let $r \in (\frac{1}{4},\frac{1}{2})$ be such that $r_i \leq r < r_{i-1}$.
Since $r < r_{i-1}$, Corollary \ref{cor0}  implies that for all $k  \leq i - 1$, $f_{k-1}(r) < f_{k}(r) $.
Apply Proposition \ref{inc1} on $f_{i-2}$ to get $f_{i-2}(r) < f_{i-2}(r_{i-1})$.
By Proposition \ref{pro1}, $$f_{i-2}(r_{i-1}) = 1-r_{i-1}.$$
It follows that $$f_{i-2}(r) < f_{i-2}(r_{i-1}) = 1 - r_i < 1 - r.$$
This means that for all $k \leq i-2$, $$f_{k-2}(r)  < 1 - r.$$ 
It follows that $ n_r > i - 2$.
Since $r_i \leq r$,  apply Proposition \ref{inc1} on $f_{i-1}$ to get  $f_{i-1}(r_i) < f_{i-1}(r)$.
Recall that by Proposition \ref{pro1}, $$f_{i-1}(r_{i}) = 1-r_{i}.$$
It follows that $$1 - r \leq 1 - r_i = f_{i-1}(r_i) < f_{i-1}(r).$$
Thus, $f_{i-1}(r) > 1 - r$.
It follows that $n_r \leq i - 1.$
Therefore, $n_r = i-1.$ This proves the proposition.
\end{proof}
%%%%%%%%%%%%%%%%%%%%%%%%
\begin{rem}\label{inf}
Proposition \ref{pro6} shows that the function in (\ref{afunc}) is constant over the intervals in the partition from (\ref{intp}). This shows that for each $r \in (\frac{1}{4},\frac{1}{2}]$, the infimum in the definition of $n_r$ as in (\ref{eq:alpha}) is attained.  
\end{rem}
A brief discussion on Geometric Rational Sequences follows. We use the results herein to obtain explicit values for the numbers $\{r_k\}_{k \geq 1}$ as in (\ref{eq:r}).
The paper \cite{B} provides results which can be applied to sequences of the form defined below.

\begin{defn}
Let $r \in (\frac{1}{4},\frac{1}{2}]$.
A Geometric Rational Sequence $\{x_n(r)\}_{n\geq0}$ is defined by setting an initial condition $x_0(r) = r$, and recursively defining
$$x_{n+1} = \frac{r}{1 - x_n }.$$
If $x_n  = 1$, then define $x_{n+1} = \infty$, $x_{n+2} = 0$, so that $x_{n+3} = r$.
\end{defn}

\begin{defn}\label{def1}
Let $r \in (\frac{1}{4},\frac{1}{2}]$.
The reduced form of a Geometric Rational Sequence $\{y_n(r)\}_{n \geq 0}$ is defined by setting an initial condition $$y_0(r) = -1 + r,$$ 
and recursively defining $$y_{n+1}(r) = -1 - \frac{r}{y_n(r)}.$$
If $y_{n}(r) = 0$, then define  $y_{n+1}(r) = \infty$, $y_{n+2}(r) =  -1$, so that $y_{n+3}(r) =  -1 + r$.
\end{defn}

\begin{rem}\label{wy}
Note that if $y_{n+3}(r) =  -1 + r$, then $y_{n+2}(r) =  -1$. Hence, $y_{n+1}(r) = \infty$ and $y_n(r) = 0$.
This occurs if and only if the sequence $\{y_n(r)\}_{n \geq 0}$ is cyclical.
\end{rem}

\begin{rem}\label{yn}
Let $i \geq 1$, and let $r \leq r_i$.
Then,
$$y_0(r) = -1 + r,$$
$$y_1(r) = -1 + \frac{r}{1 - r} = -1 + f_1(r),$$
$$y_2(r) = -1 +  \frac{r}{1 -  f_1(r)} = -1 + f_{2}(r).$$
Continuing in this way, we see that for all $k < i$ and $r \leq r_i$,
$$y_k(r) = -1 + f_k(r).$$
\end{rem}

\begin{prop}\label{cyc}
Let $m > 0$. 
Then the reduced sequence $\{y_{n}(r_m)\}_{n\geq0}$ is cyclical with cycle length $m+3$.
\end{prop}

\begin{proof}
From Remark \ref{yn}, we have that for all $k < m$, $$y_k(r_m) =-1 + f_k(r_m).$$ 
Apply Proposition \ref{pro0} to yield $$f_{k-1}(r_{m}) < f_{k}(r_m)$$ for all $k < m$.
This guarantees no repetition in the first $m-1$ terms of the reduced sequence $\{y_{n}(r_m)\}_{n\geq0}$.
By Proposition \ref{pro1},  $f_{m}(r_m)  = 1$.
It follows that
$$y_m(r_m) = f_{m}(r_m) - 1  = 1 - 1 =0.$$
Recall the reduced sequence given in Definition \ref{def1}.
Then by definition
$$y_{m+1}(r_m) = \infty$$
$$y_{m+2}(r_m) = -1$$
$$y_{m+3}(r_m) = 1 - r_m.$$
This shows that $y_{m+3}(r_m) = 1 - r_m = y_0(r_m) $. It follows that the sequence is cyclical with cycle length $m+3$.
\end{proof}

The following theorem from \cite{B} will be used. Note that there is a typographical error in the condition $\sigma^2 < 4 \gamma$. 

\begin{theor}\label{thm0}
Let $\sigma$, $\gamma \in \mathbb{R}$ with $\sigma^2 < 4 \gamma$, 
and $\theta = \arccos{\frac{\sigma}{2\sqrt{\gamma}}}$. 
A sequence satisfying $y_{n+1} = \sigma - \frac{\gamma}{y_n}$, $y_1 \in \mathbb{R}$,
has a finite or infinite number of cluster points depending on whether or not $\frac{\theta}{\pi}$ is rational. 
Moreover, when $\frac{\theta}{\pi} = \frac{k}{m} \in \mathbb{Q}$ is irreducible, the sequence takes on $m$ distinct values $y_1,y_2,\ldots,y_m$ which are thereafter repeated in this order.
\end{theor}

\begin{prop}\label{pro5}
Let $n \geq 1$. 
Let $$R_n = \frac{1}{4}\frac{1}{\cos^2(\frac{\pi}{n+3})}.$$ 
Then the reduced sequence $\{y_k(R_n)\}_{k \geq 0}$
has cycle length $n+3$.
\end{prop}

\begin{proof}
Let $n \geq  1$.
Let $$R_n = \frac{1}{4}\frac{1}{\cos^2(\frac{\pi}{n+3})}.$$ 
In reference to Theorem \ref{thm0}, the sequence
$\{y_k(R_n)\}_{k \geq 0}$ has the parameters $\sigma = -1$, and $\gamma = R_n$.
Apply Theorem \ref{thm0} with the given parameters, and set
$$\theta = \arccos{ \big( \frac{-1}{2\sqrt{R_n}} \big)} = \arccos{ \big( - \cos(\frac{\pi}{n+3}) \big)} = \frac{\pi(n + 2)}{n+3}.$$
Then
$$\frac{\theta}{\pi} \in \mathbb{Q}.$$
By Theorem \ref{thm0}, it follows that $\{y_{k}(R_n)\}_{k \geq 0}$ has cycle length $n+3$.
\end{proof}

\begin{rem}
Proposition \ref{pro5} gives a decreasing sequence of numbers $\{R_n\}_{n >0} \subset (\frac{1}{4},\frac{1}{2}]$
such that as $n$ goes to infinity $R_n \to \frac{1}{4}$. We show that these numbers correspond to the numbers $\{r_k\}_{k \geq 1}$ given in (\ref{eq:r}). These are exactly the values of the endpoints in the partition given in (\ref{intp}).
\end{rem}

\begin{prop}\label{Rn}
Recall the sequence $\{r_m\}_{m >0}$ defined in (\ref{eq:r}). Then for all $n > 0$, $$r_n = R_n.$$
\end{prop} 

\begin{proof}

Apply induction on $n$.
Let $n = 1$.
It is easy to check that $r_1 = \frac{1}{2} = R_1.$

Let $n = 2$.
Recall that $\{R_n\}_{n >0}$ is decreasing. Hence, $R_2 < R_1 = \beta_{2}$. By Proposition \ref{pro1} it follows that 
$R_2 \in \mathrm{dom}(f_2)$.
By Proposition \ref{pro5}, the sequence $\{y_{k}(R_2)\}_{k \geq 0}$ has cycle length $2+3$. 
From Remark \ref{wy}, it follows that $y_2(R_2) = 0$.
Then $$y_2(R_2) = 0 = 1 - 1 = f_2(R_2) - 1.$$ This implies that $f_2(R_2) = 1$.
Thus, $ r_2 = R_2$.

Let $n = 3$.
Note $R_3 < R_2 = \beta_{3}$. Then by Proposition \ref{pro1}, it follows that 
$R_3 \in \mathrm{dom}(f_3)$.
By Proposition \ref{pro5}, the sequence $\{y_{k}(R_3)\}_{k \geq 0}$ has cycle length $3+3$. 
From Remark \ref{wy}, it follows that $y_3(R_3) = 0$.
Then $$y_3(R_3) = 0 = 1 - 1 = f_3(R_3) - 1.$$ This implies that $f_3(R_3) = 1$.
Thus $ r_3 = R_3$.

Fix $k > 3$.
Assume that $ r_n = R_n$ for $n = 1,2,\ldots,k $.
Note $R_{k+1}< R_{k} = \beta_{k+1}$.
Then by Proposition \ref{pro1}, it follows that 
$R_{k+1} \in \mathrm{dom}(f_{k+1})$.
The sequence $\{ y_l (R_{k+1}) \}_{l \geq 0}$ has cycle length $(k+1)+3$. 
From Remark \ref{wy}, it follows that $y_{k+1}(R_{k+1}) = 0$.
Then $$y_{k+1}(R_{k+1}) = 0 = 1 - 1 = f_{k+1}(R_{k+1}) - 1.$$ This implies that $f_{k+1}(R_{k+1}) = 1$.
Thus, $ r_{k+1} = R_{k+1}$.
\end{proof}

\section{Main results}
In this section we give an upper bound for the minimal dispersion.
%
%We also give a closed form expression for the upper bound.
%
\subsection{Upper bound}
Let $r \in (\frac{1}{4},\frac{1}{2}]$, and let $n_r$ be as in (\ref{eq:alpha}). 
We define the following configurations on the diagonal
\begin{equation} \label{eq:G}
\mathfrak{q}(r)= \{f_k(r) \textbf{1}: 0\leq k \leq n_r  \}.
\end{equation}
It is clear that $|\mathfrak{q}(r)| = n_r +1$.
\\
We define two fundamental box types (Type 1 and Type 2). The definitions are given below, but first some examples in the plane should be elucidating. One can always think about the $2-$dimensional projection of a higher dimensional box onto one of its faces.\\
\newpage
\begin{center}
\center{(Example of Type 1 boxes, $d = 2$)}\\
\end{center}

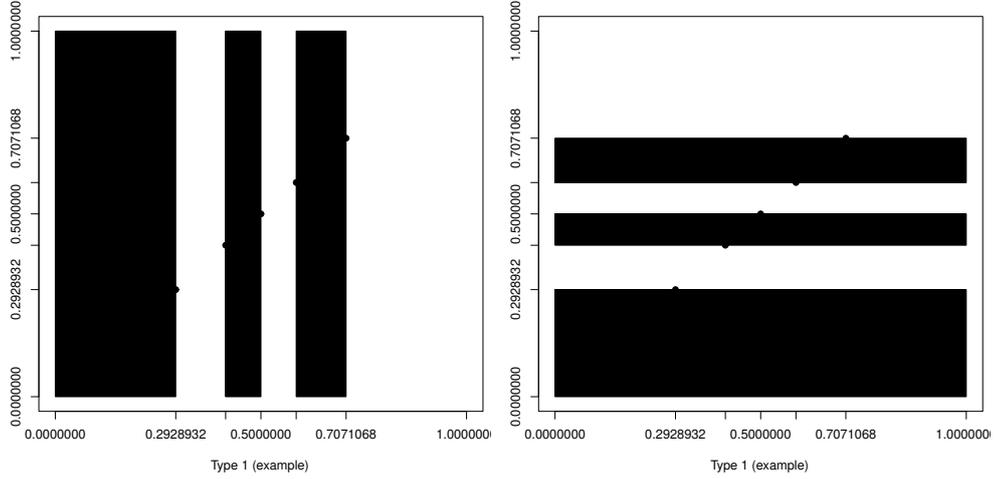
{\includegraphics[scale = .4,trim={5.925cm 1cm 5.923cm 1cm},clip]{b2.png}}
{\includegraphics[scale = .4,trim={5.922cm 1cm 5.921cm 1cm},clip]{b1.png}}\\
\begin{center}
{(Example of Type 2 boxes, $d=2$)}\\
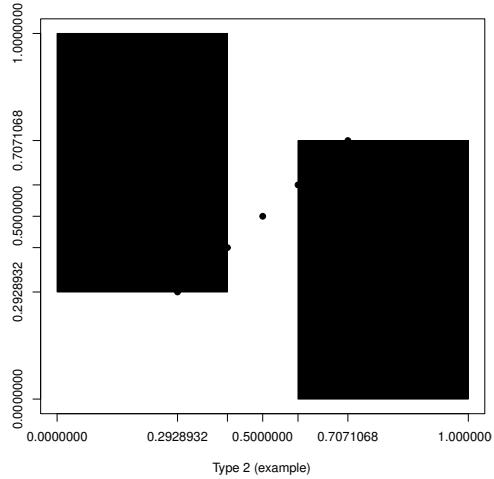
{\includegraphics[scale = .4,trim={5.95cm 1cm 5.921cm 1cm},clip]{b4.png}}

\end{center}

%{\includegraphics[scale = .4,trim={5.922cm 1cm 5.93cm 1cm},clip]{b3.png}}
%
%
\begin{defn}
Let $B = I_1 \times I_2 \times \cdots \times I_d \in \mathfrak{B}$.
The box $B$ is \textit{Type $1$}, if one of the following conditions holds.
There exists $1 \leq j \leq d$, such that $I_j \subset [0,r]$, or such that $I_j \subset [f_{n_r}(r),1]$.
There exists $1 \leq j \leq d$, and  $0\leq k \leq n_r -1$, such that $I_j \subset [f_k(r),f_{k+1}(r)]$.
%Let $B = I_1 \times I_2 \times \cdots \times I_d \in \mathfrak{B}$.
\medskip
\\
\noindent
The box $B$ is \textit{Type $2$}, if 
there exist  $1 \leq j,l \leq d$, and $0\leq k \leq n_r - 1$, such that $I_j \subset [f_k(r),1]$, and $I_l \subset [0,f_{k+1}(r)]$.
\end{defn}

\begin{lemma}\label{lem1}
Let $r \in (\frac{1}{4},\frac{1}{2}]$.
Let $B \in \mathfrak{B}$ be a box of Type $1$ or Type $2$.
Then \normalfont{Vol}$(B) \leq r$.
\end{lemma}

\begin{proof}
Let $B = I_1 \times I_2 \times \cdots \times I_d \in \mathfrak{B}$.
First assume that $B$ is Type $1$. 
Assume that $I_i \subset [0,r]$. Then the lemma trivially holds.
Assume that there exists $1 \leq i \leq d$, such that $I_i \subset [f_{n_r}(r),1]$.
Since $n_r$ is the smallest integer such that $ f_{n_r}(r) \geq 1 - r $, it follows that~$$\textrm{Vol}(B) \leq |I_i| \leq 1 - f_{n_r}(r) \leq r.$$
Assume that there exists  $1 \leq i \leq d$, and $0\leq k < n_r $ such that $I_i \subset [f_k(r),f_{k+1}(r)]$.
Then $$\textrm{ Vol}(B) \leq |I_i|  \leq f_{k+1}(r)- f_k(r)  = \frac{r}{1 - f_{k}(r)} - f_{k}(r) = \frac{r - (1 - f_{k}(r))f_k(r)}{1 - f_k(r)}.$$
Recall that $n_r$ is the smallest integer such that $1 - r  \leq f_{n_r}(r) $.  Since  $k < n_r$, it follows that $f_k(r) < 1 - r$.
Then $$\textrm{Vol}(B) \leq \frac{r - (1 - f_{k}(r))f_k(r)}{1 - f_k(r)} \leq \frac{r - rf_k(r)}{1 - f_k(r)} = r.$$ 
Hence, if $B\in \mathfrak{B}$ is Type $1$, then Vol$(B) \leq r$.

Let $B \in \mathfrak{B}$ be a Type $2$ box. By definition there exist $1 \leq i,j \leq d$ and $0\leq k \leq n_r - 1$ such that $I_i \subset [f_k(r),1]$ and  $I_j \subset [0,f_{k+1}(r)]$.
Recall the definition of $f_{k+1}$ as in (\ref{f}) given by $$f_{k+1}(r) = \frac{r}{1 - f_k(r)}.$$ It follows that $$\textrm{Vol}(B) \leq |I_i||I_j| \leq (1 - f_k(r))\frac{r}{1 - f_k(r)} = r.$$
This proves the lemma.
\end{proof}

\begin{lemma}\label{lem2}
Let $r \in (\frac{1}{4},\frac{1}{2}]$. 
Let $B  \in \mathfrak{B}$ be such that $\mathfrak{q}(r) \cap B = \emptyset$. 
Then $B$ is either Type $1$ or Type $2$.
\end{lemma}

\begin{proof}
Let
\begin{equation} \label{eq:P}
P(r) = \{p_i : p_i = f_i(r),\  0\leq i \leq n_r\},
\end{equation}

$$\mathfrak{q}(r) = P(r)\mathbf{1}.$$
Let $B = I_1 \times I_2 \times \cdots \times I_d \in \mathfrak{B}$.
Notice that $\mathfrak{q}(r) \not = \emptyset$ for all $r \in (\frac{1}{4},\frac{1}{2}]$.
Then it is clear that there exists $1 \leq i \leq d$ such that $I_i \neq [0,1]$.
Define $$Q := I_i \cap P(r).$$
If $Q = \emptyset$, then $B$ is Type $1$. From here we list the remaining possible cases.

\medskip
\noindent
Case 1: In this case $Q = \{p_{n_r}\}$.
Then $p_{n_r} \in I_i$.
It is clear that $I_i \subset (p_{{n_r}-1},1]$.
Since $B\cap \mathfrak{q}(r)= \emptyset$, there exists $1 \leq j \leq d$ such that $I_j \subset (p_{n_r},1]$ or $I_j \subset [0,p_{n_r})$.
If $I_j \subset (p_{n_r},1]$, then by definition $B$ is Type $1$.
If $I_j \subset [0,p_{n_r})$, then since $I_i \subset (p_{n_r-1},1]$, by definition $B$ is Type $2$. 

\medskip
\noindent
Case 2: Denote $Q_0 = Q$.
In the second case let $$0 <  m \leq n_r,$$ and define  $$Q_0 \subset \{p_m,p_{m+1}, \ldots ,p_{ n_r} \}.$$%.
The following algorithm shows that $B$ is  Type $1$ or Type $2$.
Let $I_{m_0} = I_i$. Then it is clear that  $I_{m_0} \subset (p_{m - 1},1]$.
Recall that $B \cap \mathfrak{q}(r) = \emptyset$. 
Then there exists $I'_{m_0}$,
such that  $I'_{m_0} \subset [0,p_{m})$ or $I'_{m_0} \subset (p_{m},1]$.
If $I'_{m_0} \cap Q_0 = \emptyset$, then $B$ is Type $1$.
If $I'_{m_0} \subset [0,p_{m})$, then since $I_{m_0} \subset (p_{m-1},1]$, $B$ is Type $2$.
If $I'_{m_0} \subset (p_{m},1]$, then denote $I_{m_1} := I'_{m_0}$.
Let $$1 < m \leq n_r ,$$ and define $$Q_1 = Q_0\cap I_{m_1} \subset \{ p_m,p_{m+1}, \ldots ,p_{ n_r} \}.$$
If $I_{m_1} \cap P(r) = \{p_{n_r}\}$, then appeal to Case 1. 
Since $B \cap \mathfrak{q}(r) = \emptyset$ there exists $I'_{m_1}$, 
such that
$I'_{m_1} \subset [0,p_{m})$ or $I'_{m_1} \subset (p_{m},1]$.
If $I_{m_1}' \cap Q_1 = \emptyset$, then $B$ is Type $1$.
If $I_{m_1}' \subset [0,p_{m})$, then since $I_{m_1} \subset (p_{m -1},1]$ $B$ is Type $2$.
If $I_{m_1}' \cap P(r) = \{p_{n_r}\}$, then appeal to Case $1$.
If $I'_{m_1} \subset (p_{m},1]$
denote $I_{m_2}: = I'_{m_1} \subset (p_{m},1]$, and continue the algorithm.
At step $\ell$ of the algorithm let $$\ell  < m \leq n_r ,$$ and define
$$Q_\ell = Q_{m_{\ell -1} }\cap I_{m_{\ell}} \subset \{ p_{m }, p_{m + 1}, \ldots, p_{n_r}\}.$$
If $Q_\ell = \emptyset$, then $B$ is Type $1$.
Assume $Q_\ell \not = \emptyset$. If $ I_{m_\ell} \cap P(r) = \{p_{n_r} \}$, then appeal to Case $1$.
If there exists an interval $I'_{m_{\ell}}$,
such that $I'_{m_{\ell}} \subset [0,p_{m})$, then since $ I_{m _{\ell}} \subset(p_{m-1},1]$ by definition $B$ is Type $2$.
The algorithm terminates after at most $n_r$ steps.

\medskip
\noindent
Case 3: In the last case $p_0 \in Q$.
Since $B \cap \mathfrak{q}(r)= \emptyset$ there exists $I_j \subset [0,p_0)$ or $I_j \subset (p_0,1]$.
If $I_j \subset [0,p_0)$, then $B$ is Type $1$.
Finally, if $I_j \subset (p_0,1]$, then apply the results in Case $1$ and Case $2$. 
This proves the lemma.
\end{proof}

\begin{cor}\label{cor1}
Let $r \in (\frac{1}{4},\frac{1}{2}]$.  Then $$N(r,d) \leq \alpha(r).$$
\end{cor}
\begin{proof}
Let $n \geq 1$. Let $r \in (\frac{1}{4},\frac{1}{2}]$ be such that $r_n \leq r < r_{n-1} $. 
Recall the configuration $\mathfrak{q}(r)$ defined as in (\ref{eq:G}). 
By Lemma \ref{lem1}, and Lemma \ref{lem2} we get that
$$\mathrm{disp}(\mathfrak{q}(r)) = r.$$
Since $|\mathfrak{q}(r)| = n_r + 1 = n$, it follows that $N(r,d) \leq n$.
\end{proof}
Corollary \ref{cor1} gives an upper bound for the minimal dispersion. 
\subsection{Formula derivation}
We derive a simple formula for the upper bound given in Corollary \ref{cor1}.
\begin{cor}\label{cor2}
Let $r \in (\frac{1}{4},\frac{1}{2}]$. 
Then
$$\alpha(r) = \left \lfloor {\frac{\pi}{\arccos{( \frac{1}{2\sqrt{r}})}}}  \right \rfloor - 3.$$
\end{cor}

\begin{proof}
Let $k>0$. 
By  Proposition \ref{pro5} and the conclusion of Remark \ref{Rn} it is clear that
\begin{equation}\label{acos}
\frac{\pi}{\arccos{ (\frac{1}{2\sqrt{r_k}})}}- 3 = \alpha(r_k) = k.
\end{equation}
Notice that the function is strictly decreasing on $(\frac{1}{4},\frac{1}{2}]$.
Let $r \in (\frac{1}{4},\frac{1}{2}]$, and let $k>0$ be such that $ r_{k}  \leq r  < r_{k-1}$. 
By Proposition \ref{pro6} and (\ref{acos}) it follows that 
$$  \alpha(r) = k = \left \lfloor {\frac{\pi}{\arccos{( \frac{1}{2\sqrt{r}})}}}  \right \rfloor - 3 .$$
\end{proof}

\begin{theor}
Let $r \in (\frac{1}{4},\frac{1}{2})$. Then $$ N(r,d) \leq \left \lfloor  \frac{\pi}{\sqrt{r - \frac{1}{4}}} \right \rfloor  - 3.$$
\end{theor}
 
 \begin{proof}
Let $r \in (\frac{1}{4},\frac{1}{2})$. Then
$${\arccos \big(\frac{1}{2 \sqrt{r}} \big)} \geq {\sqrt{r - \frac{1}{4}}}.$$
This combined with Corollary \ref{cor1} and Corollary \ref{cor2} gives
$$N(r,d) \leq \left \lfloor  \frac{\pi}{{\arccos \big(\frac{1}{2 \sqrt{r}} \big)} } \right \rfloor  -3
\leq \left \lfloor  \frac{\pi}{\sqrt{r - \frac{1}{4}}} \right \rfloor - 3.$$
This proves the theorem.
 \end{proof}

\section{Sharp estimates}
We prove some results about our configurations on the diagonal and diagonal analogues in the cube.  

\subsection{The diagonal}
\begin{prop}\label{diag0}
Let $r \in (\frac{1}{4},\frac{1}{2}]$, and let $T \subset [0,1]^d$ be on the diagonal 
such that $$\mathrm{disp}(T) = r.$$ Then $$|T| \geq n_r + 1.$$
\end{prop}

\begin{proof}
Let $r \in (\frac{1}{4},\frac{1}{2}]$.
Let $T \subset [0,1]^d$ be on the diagonal such that $$\mathrm{disp}(T) = r.$$
Let $n_r$ be as in (\ref{eq:alpha}), and let $\mathfrak{q}(r)$ be the configuration as defined in (\ref{eq:G}). 
Let $P(r)$ be as defined in (\ref{eq:P}).
Let $$E  = \{e_i \in [0,1]: e_1<e_2< \cdots <e_m, \  m =  |T| \}$$ such that $$T = \{e_i \textbf{1}: e_i \in E \}.$$
Assume toward a contradiction that $|T| \leq n_r $. 
Partition $[0,1]$ into $n_r + 2$ intervals
\begin{equation}\label{partition}
[0,p_0], (p_0,p_1], (p_1,p_2], \ldots, (p_{n_r-1},p_{n_r}],(p_{n_r},1].
\end{equation}
Since $|T| \leq n_r$, there exist at least two intervals which do not intersect $E$. 
Assume $E \cap [0,p_0] = \emptyset$.  Then $p_0 < e_1$.
Construct the box $$B = [0,e_1) \times [0,1]^{d - 1}.$$ 
Then $B \cap T = \emptyset$. However, Vol$(B) = e_1 > p_0 = r$ which contradicts the assumption that disp$(T) =  r$.

Now assume $$E \cap [0,p_0] \not = \emptyset.$$
Then $e_1 \leq  p_0$.
Remove $[0,p_0]$ from the partition in (\ref{partition}).
Then two intervals in the partition
$$ (p_0,p_1], (p_1,p_2], \ldots, (p_{n_r-1},p_{n_r}],(p_{n_r},1]$$
must not intersect $E$.
If there exists $i < n_r$ such that  $e_m < p_{i+1}$, then construct the box
$$B = [0,1] \times [p_i,1] \times [0,1]^{d - 2}.$$
It follows that $B \cap E = \emptyset$, however, $$\mathrm{Vol}(B) = (1 - p_i) > p_{i+1}(1 - p_i) = r.$$
This contradicts the assumption that disp$(T) = r$.
Let $i < n_r$ be the smallest integer such that $$E \cap (p_i,p_{i+1}] = \emptyset.$$ 
Assume $0< k \leq m$ is the smallest integer such that $ p_{i+1} < e_k$. 
Construct the box $$B = [0,e_k) \times [p_i,1] \times [0,1]^{d - 2}.$$
Since $T$ is on the diagonal $B \cap E = \emptyset$, however, $$\mathrm{Vol}(B) = e_k(1 - p_i) > p_{i+1}(1 - p_i) = r.$$
This contradicts the assumption that disp$(T) = r$.

\noindent 
The proposition follows.
\end{proof}
\begin{prop}
Let $i >0$. Let $r_i$ be as defined in (\ref{eq:r}). 
The configuration given by $\mathfrak{q}(r_i)$ in \eqref{eq:G} are symmetric on the diagonal. 
That is, if $0 \leq j \leq i-1$, then
$$1 - f_j(r_i) = f_{(i-1) - j}(r_i).$$ 
\end{prop}
\begin{proof}
Let $i >0$. Let $r_i$ be as defined in (\ref{eq:r}). 
Employ an inductive argument on $j$.
Let $j = 0$. 
Then by Proposition \ref{pro1} it follows that $$1 - f_0(r_i) = 1 - r_i = f_{i-1}(r_i).$$
Let $j = 1$.
Then by definition of the functions, and by Proposition \ref{pro1} $$ f_{i-1}(r_i)(1 - f_{i-2}(r_i)) =  \frac{r_i}{(1 - f_{i-2}(r_i))}(1 - f_{i-2}(r_i))  = r_i.$$
By Proposition \ref{pro1}, $f_{i-1}(r_i) = 1 - r_i$. It follows that 
$$  (1 - r_i)(1 - f_{i-2}(r_i)) = r_i.$$
Then
$$  1 - f_{i-2}(r_i) = \frac{r_i}{1 - r_i} = f_1(r_i),$$
in particular
$$  1 - f_1(r_i) = f_{i-2}(r_i) =  f_{(i-1)- 1}(r_i).$$
Fix $0 \leq j < i-1$, and assume the induction hypothesis,
for all $0 \leq k \leq j $. Namely that $$1 - f_k(r_i) = f_{(i-1) - k}(r_i).$$
We show that $$1 - f_{j+1}(r_i) = f_{(i-1) - ({j+1})}(r_i).$$
By the induction hypothesis 
$$  1 -  f_j(r_i) = f_{(i  - 1) - j}(r_i).$$
By construction of the functions,  we have that
$$ f_{(i -1) - {j}}(r_i)(1 - f_{(i-1) - j - 1}(r_i)) = \frac{r_i}{(1 - f_{(i-1) - j - 1}(r_i))} (1 - f_{(i-1) - j - 1}(r_i))  = r_i .$$
%Therefore,
%$$(1 -  f_j(r_i))(1 - f_{(i-1) -  j - 1}(r_i)) = r_i.$$
Therefore,
$$1 - f_{(i-1) - j - 1}(r_i) = \frac{r_i}{(1 -  f_j(r_i))}  = f_{j+1}(r_1).$$
It follows that 
$$1 -  f_{j+1}(r_1) = f_{(i-1) - (j + 1)}(r_i) .$$
The proposition follows.
\end{proof}
\subsection{Extended diagonal}
\begin{defn}
Let $ r \in ( \frac{1}{4}, \frac{1}{2}]$, and $d > 1$.
Let $P(r)$ be given by (\ref{eq:P}), and denoted as $P(r) = \{p_i : p_i = f_i(r), 0\leq i \leq n_r\}$. 
Define the Extended Diagonal as 
$$\mathfrak{D}(r,d) = [0,p_0]^d \cup (p_0,p_1]^d\cup \cdots \cup (p_{n_r},1]^d.$$
\end{defn}

\begin{defn}
Let $d > 1$. Let $x = (x_1,x_2,\ldots,x_d) \in [0,1]^d$. Define $$s(x) = \min \{x_i: 1 \leq  i \leq d\}.$$
\end{defn}

\begin{prop}\label{diag1}
Let $r \in (\frac{1}{4},\frac{1}{2}]$, and $d > 1$. 
Let $A \subset \mathfrak{D}(r,d)$ be such that $$\mathrm{disp}(A) = r.$$ Then $$|A| \geq n_r + 1.$$
\end{prop}
\begin{proof}
Assume that the hypothesis holds.
Let $n_r$ be as in (\ref{eq:alpha}), and let $\mathfrak{q}(r)$ be the configuration as in (\ref{eq:G}).
Define $$C_0 = [0,p_0]^d, \ C_{n_r+1} = (p_{n_r},1]^d.$$ 
For $0 < i \leq n_r $ define $$C_i = (p_{i-1},p_{i}]^d.$$
Assume toward a contradiction that $|A| \leq n_r $.
The $n_r$ points contained in $A$ must lie in the $n_r + 2$ disjoint sets in $\{C_i: 0 \leq i \leq n_r +1\}$ composing $\mathfrak{D}(r,d)$.
There exist at least two integers $i \leq n_r$, such that $ A \cap C_i = \emptyset$.

First assume that $ A \cap C_0  = \emptyset$. 
Since $A \subset \mathfrak{D}$, it follows that $s(p) > p_0$ for each $p \in A$.
Let 
$$t = \min \{s(p): p \in A  \}.$$
Construct the box $B = [0,t) \times [0,1]^{d-1}$. 
The magnitude of the components of each $p \in A$ is bounded below by $t$. It follows that  $ A \cap B = \emptyset$, however, Vol$(B) =  t > p_0 = r$.
This contradicts the assumption that $\mathrm{disp}(A) = r$.

Now assume that $ A  \cap C_0 \not = \emptyset$.
Let $ 1 \leq  i \leq n_r $ be the smallest number such that $ A \cap C_i = \emptyset.$
Assume that for all $j > i$, $A \cap C_j = \emptyset$.
Then set $t = 1.$
Construct a box 
$$B = [0,t]\times (p_{i-1},1] \times  [0,1]^{d-2}.$$
Since $A \subset \mathfrak{D}$, it follows that $B\cap A = \emptyset$, 
however, $$\textrm{Vol}(B) = (1  - p_{i-1}) > p_{i}(1 - p_{i-1}) = r.$$ 
This contradicts the assumption that $\mathrm{disp}(A) = r$.
Assume that $A \cap C_j \not = \emptyset$ for some smallest $j > i$.
Then set $$t = \min \{s(p):  p \in  A \cap C_j \}.$$
Construct a box 
$$B = [0,t)\times (p_{i-1},1] \times  [0,1]^{d-2}.$$
Since $A \subset \mathfrak{D}$, it follows that $ A  \cap B= \emptyset$, 
however, $$\textrm{Vol}(B) = t(1  - p_{i-1}) > p_{i}(1 - p_{i-1}) = r.$$ 
This contradicts the assumption that $\mathrm{disp}(A) = r$.
It follows that $$n_r + 1 \leq |A|.$$
\end{proof}
\begin{rem}
The condition $d > 1$ in Proposition \ref{diag1} is required. 
Let $r = \frac{1}{3}$ and $d = 1$. 
Note $\mathfrak{q}(r)= \{\frac{1}{3},\frac{1}{2},\frac{2}{3}\}$.
Define $$U_0 = [0,r], U_1 = (r,f(r)], U_2 = (f(r),1-r], U_3 = (1-r,1].$$
Then by definition $$\mathfrak{D}(r,1) = U_0 \cup U_1\cup U_2 \cup  U_3.$$
Let $A \subset \mathfrak{D}(r,1)$ be the set $\{\frac{1}{3},\frac{2}{3}\}$.
Then $$\mathrm{disp}(A) = \frac{1}{3}.$$
Since $|A|  = 2, $ and $$n_r + 1   =  |\mathfrak{q}(r)| =  3 $$ the proposition fails.
This example shows that Proposition \ref{diag1} only holds when $d > 1$.
\end{rem}
\begin{prop}\label{diag2}
Let $d \geq 2$. Let $r_n$ be as in (\ref{eq:r}) and let $\mathfrak{q}(r_n)$ be as in (\ref{eq:G}).
Let $A \subset \mathfrak{D}(r_n,d)$, be such that $|A| = n$, and
\begin{equation}\label{cond}
\mathrm{disp}(A) = r_n.
\end{equation}
Then $$A = \mathfrak{q}(r_n).$$
\end{prop}

\begin{proof}
Assume the hypothesis. Note $A \not = \emptyset$.
For all $0 \leq i \leq n $ define $C_i$ to be as in the proof of Proposition \ref{diag1}.
Assume that  for all $0 \leq i \leq n $, $A \cap C_i \not= \emptyset$. Then it follows that for all $0 \leq i \leq n-1$,
$A \cap C_i = \{p_i \mathbf{1}\}$. Hence, $A = \mathfrak{q}(r_n)$.
We state here that for all $0 \leq i \leq n-1$, $A \cap C_{i} \not= \emptyset$. 
This follows directly from the proof of Proposition \ref{diag1}: thus we shall omit the details to avoid repetition. 
Then for all $0 \leq i \leq n- 1$, $|A \cap C_i |  =  1$. 
Now apply induction on $i$ to show that for all $1 \leq i \leq n$, $A \cap C_{n- i}  = \{p_{n - i} \mathbf{1}\}$. 

Let $i  = 1$.
Assume toward a contradiction that $$A \cap C_{n-1} \not = \{ p_{n-1} \mathbf{1}\}.$$
Let $p \in A \cap C_{n-1} \setminus \{ p_{n-1} \mathbf{1}\}$. There exists a maximum component of $p$ which is less than $p_{n-1}$. Without loss of generality, assume that this is the first component. Denote the magnitude of this component as $b(p) < p_{n-1}$. 
Construct the box $$B =(b(p),1] \times  [0,1]^{d-1} .$$
Since $A \subset \mathfrak{D}$, it follows that $A \cap B  = \emptyset$.
However, $$\mathrm{Vol}(B) = 1 - b(p) > 1 - p_{n-1} = r_n.$$ 
This contradicts (\ref{cond}).
Therefore, $A \cap C_{n-1}  = \{p_{n-1} \mathbf{1}\}$. 
Fix $0  < m \leq n-1 $. 
As an induction hypothesis assume that  for all $m \leq k \leq  n-1$, $A \cap C_{k}  = \{p_k \mathbf{1}\}$. 
Assume toward a contradiction that $A \cap C_{m-1} \not = \{p_{m-1} \mathbf{1}\}$. 
Let $p \in A \cap C_{m-1} \setminus \{ p_{m-1} \mathbf{1}\}$. There exists a maximum component of $p$ which is less than $p_{m-1}$. Without loss of generality, assume that this is the first component. Denote the magnitude of this component as $b(p) < p_{m-1}$. 
Construct the box $$B =(b(p),1] \times[0,p_m) \times  [0,1]^{d-1} .$$
Since $A \subset \mathfrak{D}$, it follows that $A \cap B  = \emptyset$.
Then $$\mathrm{Vol}(B) = p_m(1 - b(p)) > p_m(1 - p_{m-1}) =  r_n.$$
This contradicts (\ref{cond}).
Therefore, it follows that $A \cap C_{m-1}  =\{ p_{m-1} \mathbf{1}\}$.
Hence, for all $0\leq k \leq n-1$, $A \cap C_{k}  = \{p_{k} \mathbf{1}\}$.
It follows that  $A = \mathfrak{q}(r_n)$.
\end{proof}
Now we show that the bound in Corollary \ref{cor2} is sharp, given that $d$ is large enough.
Recall that for each $r \in (\frac{1}{4},\frac{1}{2}]$, there exists $n > 0$ such that 
$ r_n \leq r < r_{n-1}$, and  $$n_r + 1 =  \alpha(r)  =  n.$$
\begin{theor}
Let $r \in (\frac{1}{4},\frac{1}{2}]$. 
Then
$$N(r,d) = n_r + 1,$$
given that $d \geq {n }^{n - 1} + 1$, where $n = \alpha(r)$.
\end{theor}
\begin{proof}
Let $r = \frac{1}{2} = r_1$, then for all $d \geq 2$, $$N(r,d) = 1 = n_{r_1} + 1.$$
Let $r \in (\frac{1}{4},\frac{1}{2})$ be such that $r_2 \leq r < r_1$. 
Let  $d  \geq 3$. 
Define $U_0 = [0,r_1]$, $U_1 = (r_1,1]$, and
denote ${[0,1]}^d = (U_0 \cup U_1)^d$.
Assume toward a contradiction that there exists $q_1 \in [0,1]^d$ such that $$\mathrm{disp}(\{q_1\}) =  r.$$
Then either $q_1 \in [0,1]^{d-1} \times U_0$ or $q_1 \not \in [0,1]^{d-1} \times U_0$. 
Assume $q_1 \in [0,1]^{d-1} \times U_0$. 
Construct the box $B = [0,1]^{d-1} \times U_1$. 
Then $q_1 \not \in B$, and $$\mathrm{Vol}(B) = r_1 = \frac{1}{2} > r.$$
This contradicts the assumption that $\mathrm{disp}(\{q_1\}) = r.$
Assume  $q_1 \not \in [0,1]^{d-1} \times U_0$. 
Construct the box 
$B = [0,1]^{d-1} \times U_0$.
Then $q_1 \not \in B$, however, $$\mathrm{Vol}(B)  = r_1>r.$$
This contradicts the assumption that $\mathrm{disp}(\{q_1\}) = r.$
Then
$ 1 < N(r,d),$
and by Corollary \ref{cor1}, $ N(r,d) \leq  2.$ 
It follows that
$$ N(r,d) =  2.$$

Let $r$ be such that  $ r_3 \leq r < r_2 $. Since $\alpha(r) = 3$, $d \geq 10$.
Define $U_0 = [0,r_2]$, $U_1 = (r_2,1-r_2]$, $U_2 = (1-r_2,1]$. 
Select two points $q_1, q_2 \in [0,1]^d$. 
Assume toward a contradiction that $\mathrm{disp}(\{q_1,q_2\})  = r$.
The components of $q_1$ and $q_2$ are contained in the intervals $U_0,U_1,U_2$.
Denote the components of $q_1,q_2$ as $\{q_{1,i}\}_{1 \leq i \leq d}$, $\{q_{2,i}\}_{1 \leq i \leq d}$. 
Let $M_1 \geq 4$ denote the largest number of components of $\{q_{1,i}\}_{1\leq i\leq d}$ contained in a single interval, denoted as $U_{m_1}$.
Denote the corresponding indices as $\{a_i\}_{1\leq i \leq {M_1}}: = \{a_i: a_1< a_2 < \cdots~<~a_{M_1}~\}.$
Let $M_2 \geq 2$ denote the largest number of components of $\{q_{2,a_i}\}_{1 \leq i \leq {M_1}}$ contained in a single interval, denoted as $U_{m_2}$.
Denote the corresponding indices as $\{b_i :1 \leq i \leq M_2 \} \subset \{a_i\}_{1 \leq i \leq M_1}$.
It follows that $$q_{1,b_1},q_{1,b_2} \in U_{m_1},$$
$$q_{2,b_1},q_{2,b_2} \in U_{m_2}.$$ 
Project onto the components,
$$q_1 \to (0,0,\ldots,0,q_{1,b_1},0,0,\ldots,0,q_{1,b_2},0,0,\ldots,0) = q_1',$$
and
$$q_2 \to (0,0,\ldots,0,q_{2,b_1},0,0,\ldots,0,q_{2,b_2},0,0,\ldots,0) = q_2'.$$
The points are projected onto a $2$-dimensional face of $[0,1]^{d}$,
given by $$\{0\}^{b_1 - 1}\times[0,1] \times\{0\}^{ b_2 -(1 + b_1)}\times [0,1]\times \{0\}^{d - b_2}.$$
Note that
$$q_1' \in \{0\}^{b_1 - 1}\times U_{m_1} \times\{0\}^{ b_2 -(1 + b_1)}\times U_{m_1}\times \{0\}^{d - b_2},$$
and
$$q_2' \in \{0\}^{b_1 - 1}\times U_{m_2} \times\{0\}^{ b_2 -(1 + b_1)}\times U_{m_2}\times \{0\}^{d -  b_2}.$$
The components $q_1', q_2'$ are contained in $\mathfrak{D}(r_2,2)$.
By Proposition \ref{diag2} there exists $B' \in \mathfrak{B}$ such that
$$B' = \{0\}^{b_1 - 1}\times I_1 \times\{0\}^{b_2 - (1 + b_1)}\times I_2 \times \{0\}^{d -  b_2},$$ 
where $q_1',q_2' \not \in B'$. However, $\mathrm{Vol}(I_1 \times I_2)\geq r_2$.
Let $$B = [0,1]^{b_1 - 1 }\times I_1 \times [0,1]^{b_2 - (1 + b_1)}\times I_2\times [0,1]^{d -   b_2}.$$ 
It is clear that $q_1,q_2 \not \in B$.
However,  $\mathrm{Vol}(B) \geq r_2 > r$. 
This contradicts the assumption that $\mathrm{disp}(\{q_1,q_2\}) = r$.
Then $ N(r,d)> 2$, and
by Corollary \ref{cor1}, $ N(r,d) \leq  3.$ 
It follows that $$  N(r,d) = 3.$$
Fix $n > 3$, and let $r_{n+1} \leq r < r_n$. Since $\alpha(r) = n + 1$, $$d \geq {({n} + 1)} ^{{n}} + 1.$$ 
Let $q_1,q_2,\ldots,q_n$ be arbitrary points in the cube. 
Assume toward a contradiction that $\mathrm{disp(\{q_1,q_2,\ldots,q_n\})} = r.$
Define the partition $$U_0 = [0,r_n], U_1 =(r_n ,f_1(r_n)],  \ldots, U_{n+1} = (1-r_n,1].$$
 For all $1 \leq i \leq n$, denote the components of $q_i$ as $(q_{i,1},q_{i,2},\ldots,q_{i,d})$.
Denote $d_1 := {({n} + 1)} ^{{n}} + 1$.
Let $$M_1 \geq d_2:=  \frac{d_1 - 1}{({n} + 1)}  + 1$$ denote the largest number of components of
$\{q_{1,i}\}_{0<i\leq d}$ which are contained in a single interval, denoted as $U_{m_1}$. 
Denote the corresponding indices as $$\{a_{1,i}\}_{1 \leq i \leq M_1}.$$
Let $$M_2 \geq  d_{3} := \frac{d_{2} - 1}{({n} + 1)}  + 1$$ denote the largest number of components of
 $\{q_{2,a_{1,i}}\}_{1\leq i \leq M_1}$ which  are contained in a single interval, denoted as $U_{m_2}$.
Denote the corresponding indices as $\{a_{2,i}\}_{1 \leq i \leq M_2}$. 

\noindent
For all $1 \leq k \leq  {n}$
let $$M_{k } \geq  d_{k+1}: = \frac{d_{k } - 1}{({n} + 1)}  + 1$$ denote the largest number of components of
$\{q_{k,a_{{k-1},i}}\}_{1\leq i \leq M_{k-1}}$ which are contained in $U_{m_k}$.
Denote the corresponding indices as $\{a_{k,i}\}_{1 \leq i \leq M_k}$.
This guarantees that $M_{n}\geq 2$.
Define a projection on $1 \leq j \leq n$, such that $$q_j \to q_{j,a_{n,1}},q_{j,a_{n,2}} \in U_{m_j}.$$
This embeds into $\mathfrak{D}(r_n,2)$.
Then by Proposition \ref{diag2}, there exists a box $B'$ such that $q_j' \not \in B'$ with
$I_1,I_2$, such that 
$$\mathrm{Vol}(I_1 \times I_2) \geq r_n.$$
This can be extended to a box 
$$B = [0,1]^{a_{n,1} - 1} \times I_1 \times [0,1]^{ a_{n,2} - (1 + a_{n,1})}\times I_2\times [0,1]^{d_1 -  a_{n,2}}.$$ 
It is clear that $q_1,q_2,\ldots,q_n \not \in B$ and that $\mathrm{Vol}(B) \geq r_n > r$. 
This contradicts the assumption that $\mathrm{disp}(\{q_1,q_2,\ldots,q_n\}) = r$.
Then $ n < N(r,d) $,
and by Corollary \ref{cor1}, $ N(r,d) \leq n+1.$
It follows that $$  N(r,d) =  n+1.$$
\end{proof}

\section{Concluding remarks}
When $r = \frac{1}{4}$, and $d$ is small the following configurations are better than the best known bound which asymptotically is $\log(d)$. We present a configuration which is easy to describe and visualize.
\begin{prop}
Let $r \geq \frac{1}{4}$. Then $N(r,d) \leq 2d$.
\end{prop}
 
\begin{proof}
Let $d = 2$ let 
$$K = \big \{(\frac{1}{2},\frac{1}{4}),(\frac{1}{2},\frac{3}{4}),(\frac{1}{4},\frac{1}{2}),(\frac{3}{4},\frac{1}{2}) \big \}.$$
Every box $B \cap K = \emptyset$ inside $[0,1]^2$  is contained in one of the following\\
%%%%%%%%\begin{}
\begin{equation}
   \begin{matrix} \label{box}
   [0,\frac{1}{2}) \times [0,\frac{1}{2}), & [0,\frac{1}{2}) \times (\frac{1}{2},1], & (\frac{1}{2},1]\times [0,\frac{1}{2}), & (\frac{1}{2},1] \times (\frac{1}{2},1]\\
   
   (\frac{1}{4},\frac{3}{4}) \times (\frac{1}{4},\frac{3}{4}), & [0,\frac{1}{4}) \times [0,1], & (\frac{3}{4},1] \times [0,1], & [0,1] \times ( \frac{3}{4},1] \\
   
   [0,1] \times [0,\frac{1}{4}) & (\frac{1}{4},\frac{1}{2}) \times [0,1] &  [0,1] \times (\frac{1}{4},\frac{1}{2}) & (\frac{1}{2},\frac{3}{4}) \times [0,1]\\
   [0,1] \times (\frac{1}{2},\frac{3}{4})
   
   \end{matrix} 
\end{equation}
This gives the result in the $2$ dimensional case.
Let $d > 2$. 
Let 
$$K_1 = \bigg\{ \bigg( \frac{1}{2},\frac{1}{2},\ldots, \frac{1}{2},M_i, \frac{1}{2}, \ldots , \frac{1}{2} \bigg): 
M_i = \frac{1}{4},\  1 \leq i \leq d \bigg\}$$
$$ K_2 = \bigg\{ \bigg(\frac{1}{2},\frac{1}{2},\ldots, \frac{1}{2},M_i, \frac{1}{2}, \ldots , \frac{1}{2} \bigg)
: M_i =  \frac{3}{4},\  1 \leq i \leq d  \bigg\}.$$
Let $K = K_1 \cup K_2$. 
Then each box in $B \in [0,1]^d$ such that $B \cap K = \emptyset$ is contained in a product of $d - 2$ intervals $[0,1]$ with one of the boxes in (\ref{box}).
Each of the boxes have volume $\frac{1}{4}$. Therefore, $$\mathrm{disp}(K) = \frac{1}{4},$$ and $|K| = 2d.$
\end{proof}

\subsection{Acknowledgments}
The author would like to thank his supervisor A.E. Litvak for introducing him to the dispersion problem, for his valuable  expertise, and helpful comments while the author was working on this note. The author would also like to thank the reviewers for their comments and suggestions.

\address
\end{document}